\documentclass[a4paper,11pt]{amsart}
\usepackage{a4wide}
\usepackage{hyperref}
\usepackage[pdftex]{graphicx}
\usepackage{url}
\usepackage{enumerate}
\usepackage{amsxtra}
\usepackage{amstext}
\usepackage{amssymb}
\usepackage{tikz}
\usepackage{latexsym}
\usepackage{dsfont}
\newcommand{\pse}{\psi'({\eta})}

\newcommand{\Nat}{\mathbb N}

\newcommand{\R}{\mathbb R}
\newcommand{\Esp}{\mathbb E}
\newcommand{\vs}{\varsigma}
\newcommand{\p}{\mathbb P}

\newcommand{\lop}{\log^+}
\newcommand{\supt}{\sup_{t\geq0}}
\newcommand{\supn}{\sup_{n>0}}

\newcommand{\1}{\mathbf{1}\!}

\newcommand{\ext}{\mathrm{Ext}}

\newcommand{\intpos}{\int_0^\infty}
\newtheorem{prop}{Proposition}[section]

\newtheorem{lem}[prop]{Lemma}
\newtheorem{thm}[prop]{Theorem}
\newtheorem{rem}[prop]{Remark}
\author{Mathieu Richard}

\address{Laboratoire de Probabilit\'es et Mod\`eles Al\'eatoires, UMR 7599, Univ. Paris 6 UPMC, Case courrier 188,
4, Place Jussieu, 75252 PARIS Cedex 05, France.   \href{mailto:mathieu.richard@upmc.fr}{\nolinkurl{mathieu.richard@upmc.fr}}}
\title[Limit theorems for supercritical branching processes with immigration]{Limit theorems for supercritical age-dependent branching processes with neutral immigration}
\begin{document}\maketitle 
\begin{abstract} We consider a branching process with Poissonian immigration where individuals have inheritable types. At rate $\theta$, new individuals singly enter the total population and start a new population which evolves like a supercritical, homogeneous, binary Crump-Mode-Jagers process: individuals have i.i.d. lifetimes durations (non necessarily exponential) during which they give birth independently at constant rate $b$.
First, using spine decomposition, we relax previously known assumptions required for a.s. convergence of total population size.
 Then, we consider three models of structured populations: either all immigrants have a different type, or types are drawn in a  discrete spectrum or in a continuous spectrum. In each model, the vector $(P_1,P_2,\dots)$ of relative abundances of surviving families converges a.s. In the first model, the limit is the GEM distribution with parameter $\theta/b$.
\end{abstract}

\section{Introduction}
We want to study and give some properties about several birth-death and immigration models where immigration is structured. All individuals behave independently of the others, their lifetimes are i.i.d. but non-necessarily exponential and each individual gives birth at constant rate $b$ during her life. We will consider the supercritical case i.e., the mean number of children of an individual is greater than 1.
In the absence of immigration, if $X(t)$ denotes the number of extant individuals at time $t$, the process $(X(t),t\geq0)$ is a particular case of Crump-Mode-Jagers (or CMJ) processes \cite{Jagers_BP_with_bio}, also called general branching processes. Here, $X$ is a binary (births arrive singly) and homogeneous (constant birth rate) CMJ process.
Now, we assume that at each arrival time of a Poisson process with rate $\theta$, a new individual enters the population and starts a new population independently of the previously arrived ones. This immigration model extends to general lifetimes the mainland-island model of S. Karlin and J. McGregor \cite{Kar_McGreg67}. In that case, the total population process $X$ is a linear birth-and-death process with immigration. For more properties about this process, see \cite{Tavaré_2} or \cite{Watterson1974a} and references therein. In the context of ecology \cite{Caswell1976}, this model can be used as a null model of species diversity, in the framework of the neutral theory of biodiversity \cite{Hubbell}.

We first give the asymptotic behavior of the process $(I(t),t\geq0)$ representing the total number of extant individuals on the island at time $t$. Specifically, there exists $\eta>0$ (the Malthusian parameter associated with the branching process $X$) such that $e^{-\eta t}I(t)$ converges almost surely. S. Tavar\'e \cite{Tavaré_BirthProcImm} proved this result in the case of a linear birth process with immigration. The case of general CMJ processes was treated by P. Jagers \cite{Jagers_BP_with_bio} under the hypothesis that the variance of the number of children per individual is finite. We manage to relax this assumption in the case of homogeneous (binary) CMJ-processes thanks to spine decomposition of splitting trees \cite{Geiger1997,Amaury_cours_Mexique,Amaury_contour_splitting_trees} which are the genealogical trees generated by those branching processes. In passing, we obtain technical results on the log-integrability of $\sup_t e^{-\eta t}X(t)$.

Then, we consider models where individuals bear clonally inherited types. They intend to model a metacommunity (or mainland) which delivers immigrants to the island as in the theory of island biogeography \cite{McArthur_Wilson}. However, we made specific assumptions about the spectrum of abundances in the metacommunity. In Model I,  there is a discrete spectrum with zero macroscopic relative abundances: when an immigrant enters the population, it is each time of a new type. In Model II, we consider a discrete spectrum with nonzero macroscopic relative abundances: the type of each new immigrant is chosen according to some probability $(p_i,i\geq1)$.
In Model III, we consider a continuous spectrum of possible types but to enable a type to be chosen several times from the metacommunity, we change the immigration model:
 at each immigration time, an individual belonging to a species with abundance in $(x,x+dx)$ is chosen with probability $\frac{xf(x)}{\theta}dx$ (where $f$ is a positive function representing the abundance density and such that $\theta:=\intpos xf(x)dx<\infty$) and it starts an immigration process with immigration rate $x$. The particular case of abundance density  $f(x)=\frac{e^{-ax}}{x}$ appears in many papers.  I. Volkov et al. \cite{12944964} and G. Watterson \cite{Watterson1974} consider it as a continuous equivalent of the logarithmic series distribution proposed by R. Fisher et al. \cite{Fisher1943} as a species abundance distribution. In this particular case, species with small abundances are often drawn but they will have a small immigration rate.

In the three models, we get results for the abundances $P_1,P_2,\dots$ of different types as time $t$ goes to infinity: the vector $(P_1,P_2,\dots)$ rescaled by the total population size converges almost surely. More precisely, in Model I which is an extension of  S. Tavar\'e's result \cite{Tavaré_BirthProcImm} to general lifetimes, we consider the abundances of the surviving families ranked by decreasing ages and the limit follows a \emph{GEM distribution} with parameter $\theta/b$.
This distribution appears in other contexts: P. Donnelly and S. Tavar\'e \cite{Donnelly1986} proved that for a sample of size $n$ whose genealogy is described  by a Kingman coalescent with mutation rate $\theta$, the frequencies of the oldest, second oldest, etc. alleles converge in distribution
as $n\rightarrow\infty$ to the GEM distribution with parameter $\theta$; S. Ethier \cite{Ethier1990} showed that it is also the distribution of the frequencies of the alleles ranked by decreasing ages in the stationary infinitely-many-neutral-alleles diffusion model.

In a sense, the surviving families that we consider in our immigration model are "large" families because their abundances are of the same order as the population size. A. Lambert \cite{Amaury-Immig-Mut} considered "small" families: he gave the joint law of the number of species containing $k$ individuals, $k=1,2,\dots$

In Section \ref{prelimin}, we describe the models we consider and state results we prove in other sections. Section \ref{preuvepropCMJ} is devoted to proving a result about the process $X$, Section \ref{preuvetheoprincipal} to proving a property of the immigration process $(I(t),t\geq0)$ while in Section \ref{other_proofs}, we prove theorems concerning the relative abundances of types in the three models.

\section{Preliminaries and statement of results}\label{prelimin}
We first define splitting trees which are random trees satisfying:
\begin{itemize}
\item  individuals behave independently from one another and have i.i.d. lifetime durations,
\item conditional on her birthdate $\alpha$ and her lifespan $\zeta$, each individual reproduces according to a Poisson point process on $(\alpha,\alpha+\zeta)$ with intensity $b$,
\item births arrive singly.
\end{itemize}
We denote the common distribution of lifespan $\zeta$ by $\Lambda(\cdot)/b$ where $\Lambda$ is a positive measure on $(0,\infty)$ with mass $b$ called the \emph{lifespan measure} \cite{Amaury_contour_splitting_trees}.

The total population process $(X(t),t\geq0)$ belongs to a large class of processes called Crump-Mode-Jagers or CMJ processes. In these processes, also called general branching processes \cite[ch.6]{Jagers_BP_with_bio}, a typical individual reproduces at ages according to a random point process $\xi$ on $[0,\infty)$ (denote by $\mu:=\Esp[\xi]$ its intensity measure) and it is alive during a random time $\zeta$.
 Then, the CMJ-process is defined as
 $$X(t)=\sum_{x}\1_{\{\alpha_x\leq t<\alpha_x+\zeta_x\}},\ \ t\geq0$$
 where for any individual $x$, $\alpha_x$ is her birth time and $\zeta_x$ is her lifespan.
In this work, the process $X$ is a homogeneous (constant birth rate) and binary CMJ-process and we get $$\mu(dx)=dx\int_{[x,\infty)}\!\Lambda(dr).$$

We assume that the mean number of children per individual $m:=\int_{(0,\infty)}r\Lambda(dr)$ is greater than 1 (supercritical case).

 For $\lambda\geq0$, define $\psi(\lambda):=\lambda-\int_{(0,\infty)}(1-e^{-\lambda r})\Lambda(dr)$. The function $\psi$ is convex, differentiable on $(0,\infty)$, $\psi(0^+)=0$ and $\psi'(0^+)=1-\int_0^\infty r\Lambda(dr)<0$.
 Then there exists a unique positive real number $\eta$ such that $\psi(\eta)=0$. It is seen by direct computation that this real number is a \emph{Malthusian parameter} \cite[p.10]{Jagers_BP_with_bio}, i.e. it is the finite positive solution of $\int_0^\infty e^{-\eta r}\mu(dr)=1$ and is such that $X(t)$ grows like $e^{\eta t}$ on the survival event (see forthcoming Proposition \ref{propCMJ}).
From now on, we define $$c:=\psi'(\eta)$$ which is positive because $\psi$ is convex.

Another branching process appears in splitting trees: if we denote by $\mathcal{Z}_n$ the number of individuals belonging to generation $n$ of the tree, then $(\mathcal{Z}_n,n\geq0)$ is a Bienaym\'e-Galton-Watson process started at 1 with offspring generating function $$f(s):=\int_{(0,\infty)}b^{-1}\Lambda(dr)e^{-br(1-s)}\quad 0\leq s\leq1.$$

To get results about splitting trees and CMJ-processes, A. Lambert \cite{Amaury_cours_Mexique,Amaury_contour_splitting_trees} used tree contour techniques. He proved that the contour process $Y$ of a splitting tree was a \emph{spectrally positive} (i.e. with no negative jumps) L\'evy process whose Laplace exponent is $\psi$. Lambert obtained result about the law of the population in a splitting tree alive at time $t$. If $\tilde\p_x$
denotes the law of the process $(X(t),t\geq0)$ conditioned to start with a single ancestor living $x$ units of time,
\begin{equation}\label{prop1loideX(t)}\tilde\p_x(X(t)=0)=W(t-x)/W(t)\end{equation}
and conditional on being nonzero, $X(t)$ has a geometric distribution with success probability $1/W(t)$
i.e. for $n\in\Nat^*$,
\begin{equation}\label{prop2loideX(t)}\tilde\p_x\left(X(t)=n\right)=\left(1-\frac{W(t-x)}{W(t)}\right)\left(1-\frac{1}{W(t)}\right)^{n-1}\frac{1}{W(t)}\end{equation}
 where $W$ is the \emph{scale function} \cite[ch.VII]{Levy_processes} associated with $Y$: this is the unique absolutely continuous increasing function $W:[0,\infty]\rightarrow [0,\infty]$ satisfying
\begin{equation}\label{defW}\intpos e^{-\lambda x}W(x)dx=\frac{1}{\psi(\lambda)}\quad \lambda>\eta.\end{equation}
The two-sided exit problem can be solved thanks to this scale function:
\begin{equation}\label{doubleexit}\p\left(T_0<T_{(a,+\infty)}|Y_0=x\right)=W(a-x)/W(a),\ \ 0<x<a
\end{equation}
where for a Borel set $B$ of $\R$, $T_B=\inf\{t\geq0,\ Y_t\in B\}$.

We now give some properties, including asymptotic behavior, about the CMJ-process $X$.
\begin{prop}\label{propCMJ}We denote by {\rm Ext} the event $\displaystyle\left\{\lim_{t\rightarrow\infty}X(t)=0\right\}$.\begin{enumerate}[(i)]\item We have
\begin{equation}\label{probaextinction}
\p({\rm Ext})=1-\eta/b
\end{equation}
and conditional on ${\rm Ext}^c$,
\begin{equation}\label{premier_lemme}
e^{-\eta t}X(t)\overset{\textrm{a.s.}}{\underset{t\rightarrow\infty}\longrightarrow} E\end{equation} where $E$ is an exponential random variable with parameter $c$.
\item If, for $x>0$, $\lop x:=\log x\vee0$,
\begin{equation}\label{logsupXt}
\Esp\left[\left.\left(\log^+ \sup_{t\geq0}\left(e^{-\eta t}X(t)\right)\right)^2\right|\ext^c\right]<\infty\end{equation}
\end{enumerate}
\end{prop}
We remind that $c=\pse$ and then depends on the measure $\Lambda(\cdot)$.
The proof of the last assertion requires involved arguments using spine decomposition of splitting trees. Proposition \ref{propCMJ}, which will be proved in Section \ref{preuvepropCMJ}, is known in a particular case: if the lifetime $\Lambda(\cdot)/b$ has an exponential density with parameter $d$, $(X(t),t\geq0)$ is a Markovian birth and death process with birth rate $b$ and death rate $d<b$. In that case, $\eta=b-d$, $c=1-d/b=\p(\ext^c)$ and integrability of $\supt(e^{-\eta t}X(t))$ stems from Doob's maximal inequality.\\

We now define the immigration model: let $\theta$ be a positive number and $0=T_0<T_1<T_2<\cdots$ be the points of a Poisson process of
rate $\theta$. At each time $T_i$, we assume that a new individual immigrates and starts a new population whose size evolves like $X$, independently of the other populations. That
is, if for $i\geq1$ we call $(Z^i(t),t\geq 0)$ the $i$th oldest
family (the family which was started at $T_i$) then
$Z^{i}(t)=X_{i}(t-T_i)\1_{\{t\geq T_i\}}$ where $X_1,X_2,\dots$ are copies of $X$ and $(X_i,i\geq1)$ and $(T_i,i\geq1)$ are
independent.
This immigration model is a generalization of S. Karlin and J. McGregor's model \cite{Kar_McGreg67} in the case of general lifetimes.

\begin{figure}[!ht]
\begin{center}
\includegraphics{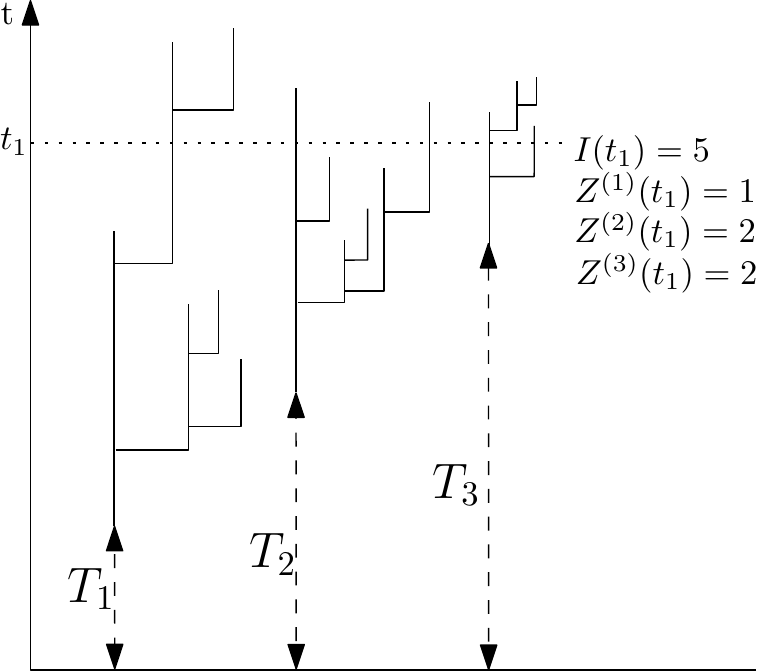}
\caption{splitting trees with immigration. The vertical axis is time, horizontal axis shows filiation. At time $t_1$, three populations are extant.}
\end{center}
\end{figure}

For $i\geq1$, define $(Z^{(i)}(t),t\geq0)$ as the $i$th oldest family among the surviving populations and $T^{(i)}$ its birthdate. In particular, by thinning of Poisson point process, $(T^{(i)},i\geq1)$ is a Poisson point process with parameter $\theta\eta/b$ thanks to (\ref{probaextinction}).

We are now interested in the joint behavior of the surviving families $(Z^{(i)}(t),t\geq0)$ for $i\geq1$:
{\setlength\arraycolsep{2pt}
\begin{eqnarray}
e^{-\eta t}Z^{(i)}(t)&=&e^{-\eta T^{(i)}}e^{-\eta \left(t-T^{(i)}\right)}Z^{(i)}(t)\nonumber\\
&\overset{(d)}=&e^{-\eta T^{(i)}}e^{-\eta \left(t-T^{(i)}\right)}X_{(i)}(t-T^{(i)})\1_{\{T^{(i)}\leq t\}}\nonumber
\end{eqnarray}}
As in (\ref{premier_lemme}), denote by
$E_i:=\lim_{t\rightarrow\infty}e^{-\eta t}X_{(i)}(t)$ for $i\geq1$. Thus
$E_1,E_2,\dots$ are i.i.d. exponential r.v. with parameter $c$.
Moreover, the sequences $(E_i,i\geq1)$ and $(T^{(i)},i\geq1)$ are
independent.
It follows that $e^{-\eta t}Z^{(i)}(t)\rightarrow e^{-\eta T^{(i)}}E_i$ a.s. as $t\rightarrow\infty$. We record this in the following
\begin{prop}\label{conv pop}
$$e^{-\eta t}(Z^{(1)}(t),Z^{(2)}(t),\dots)\underset{t\rightarrow\infty}\longrightarrow(e^{-\eta T^{(1)}}E_1,e^{-\eta T^{(2)}}E_2,\dots)  \quad\textrm{a.s.}$$ where the $E_i$'s are independent copies of $E$ and independent of the $T^{(i)}$'s.
\end{prop}

For $t\geq0$, let $I(t)$ be the size of the total population at time $t$
$$I(t)=\sum_{i\geq1}Z^i(t).$$
The process $(I(t),t\geq0)$ is a non-Markovian continuous-time branching process with immigration.
\begin{thm}\label{convpresquesureimmig}
\begin{enumerate}[(i)]
\item For $t$ positive, $I(t)$ has a negative binomial distribution with parameters $1-W(t)^{-1}$ and $\theta/b$.
 i.e. for $s\in[0,1]$, its generating function is
$$G_t(s):=\Esp\left[s^{I(t)}\right]=\left(\frac{W(t)^{-1}}{1-s(1-W(t)^{-1})}\right)^{\theta/b}.$$

\item We have $$I:=\lim_{t\rightarrow\infty}e^{-\eta t}I(t)=\sum_{i\geq1}e^{-\eta T^{(i)}}E_i  \textrm{ a.s.}$$ and $I$ has a Gamma distribution $\Gamma(\theta/b,c)$ i.e the density of $I$ with respect to Lebesgue measure is $$g(x)=\frac{c^{\theta/b}x^{\theta/b-1}e^{-c x}}{\Gamma(\theta/b)},\ x>0.$$
 \end{enumerate}
\end{thm}
The result (i) is a generalization of a result by D.G. Kendall \cite{Kendall1948}
which was the particular Markovian case of a birth, death and immigration process. The proof we give in Section \ref{preuvetheoprincipal} uses equations (\ref{prop1loideX(t)}) and (\ref{prop2loideX(t)}) about the law of $X(t)$.

There exist other proofs of the almost sure convergence in (ii), but they require stronger assumptions. For example, P. Jagers \cite{Jagers_BP_with_bio} gives a proof for the convergence of general branching processes with immigration under the hypothesis that the variance of the number of children per individual $\xi(\infty)$ is finite. In our case, this is only true if $\int_{(0,\infty)} r^2\Lambda(dr)<\infty$. In the particular Markovian case described previously, the proof is also easier since $(e^{-\eta t}X(t),t\geq0)$ is a non-negative martingale \cite[p.111]{Athreya_Ney}, $(e^{-\eta t}I(t),t\geq0)$ is a non-negative submartingale and both converge a.s.
 In the proof we give in Section \ref{preuvetheoprincipal}, the only assumption we use about the measure $\Lambda$ is that its mass is finite. The proof is based on Proposition \ref{propCMJ}(ii).

In the following, we will consider different kinds of metacommunity where immigrants are chosen and will give results about abundances of surviving populations.
 In Model I, there is a discrete spectrum with zero macroscopic relative abundances: when a new family is initiated, it is of a type different from those of previous families.
The following theorem yields the asymptotic behaviors of the fractions of the surviving subpopulations ranked by decreasing ages in the total population:
\begin{thm}[Model I]\label{thm principal splitting trees}
$$\lim_{t\rightarrow\infty}I(t)^{-1}(Z^{(1)}(t),Z^{(2)}(t),\dots)=(P_1,P_2,\dots)\textrm{\quad a.s.}$$
where the law of $(P_1,P_2,\dots)$ is a GEM distribution with parameter $\theta/b$.
 In other words, for $i\geq1$ $$P_i\overset{(d)}=B_i\prod_{j=1}^{i-1}(1-B_j)$$ and $(B_i)_{i\geq1}$ is a sequence of i.i.d. random variables with law Beta($1$,$\theta/b$) whose density with respect  to Lebesgue measure is $$\frac{\theta}{b}(1-x)^{\theta/b-1}\1_{[0,1]}(x).$$
\end{thm}
This result was proved by S. Tavar\'e \cite{Tavaré_BirthProcImm} in the case where $\Lambda(dr)=\delta_\infty(dr)$ (pure birth process); it is the exponential case defined previously with $b=1$ and $d=0$. His result is robust because we see that in our more general case, the limit distribution does not depend on the lifespan distribution but only on the immigration-to-birth ratio $\theta/b$.
In biogeography, a typical question is to recover data about population dynamics (immigration times, law of lifespan duration) from the observed diversity patterns. In this model, we see that there is a loss of information about the lifespan duration. However, the ratio $\theta/b$ can be estimated thanks to the species abundance distribution.
We will prove Theorem \ref{thm principal splitting trees} in Subsection \ref{bloup}.

In Model II, we consider a discrete spectrum with nonzero macroscopic relative abundances. Contrary to Model I where types were always new, they are now given a priori and types of immigrants are independently drawn according to some probability $p=(p_i,i\geq1)$. When a population is initiated (i.e. at each time of the $\theta$-Poisson point process), it is of type $i$ with probability $p_i>0$.

\begin{thm}[Model II]\label{thm multitype}
For $ i\geq 1$, denote by $I_i(t)$ the number of individuals of type $i$ at time $t$  and set $\theta_i:=\frac{\theta p_i}{b}$.
Then
$$\lim_{t\rightarrow\infty}I(t)^{-1}(I_1(t),I_2(t),\dots)=(P'_1,P'_2,\dots)\textrm{\quad a.s.}$$
where for $i\geq1$ $$P'_i\overset{(d)}=B'_i\prod_{j=1}^{i-1}(1-B'_j)$$ and $(B'_i)_{i\geq1}$ is a sequence of independent random variables such that $$B'_i\sim \textrm{Beta}\left(\theta_i,\frac{\theta}{b}\sum_{j\geq i+1}p_j\right).$$

In particular, for $i\geq 1$, $P'_i$ has a Beta distribution $B(\theta_i,\theta/b-\theta_i)$.
\end{thm}
The proof of this theorem will be done in  Subsection \ref{bloup2}. In this model, the limit only depend on $\theta/b$ and the metacommunity spectrum $(p_i,i\geq1)$.
\begin{rem} If the number of possible types $n$ is finite,
then $$\sum_{i=1}^{n}P'_i=\sum_{i=1}^n\frac{I_i}{I}=1,\ \
\sum_{i=1}^{n}{\theta}_i=\frac{\theta}{b}$$ and the joint density of
$(P'_1,\dots,P'_n)$ is
$$\frac{\Gamma(\theta/b)}{\prod_{i=1}^n\Gamma(\theta_i)} \left(\prod_{i=1}^{n-1} x_i^{\theta_i - 1}\1_{\{x_i>0\}}\right)(1-x_1-\cdots-x_{n-1})^{\theta_{n}-1}\1_{\{x_1+\dots+x_{n-1}<1\}}.$$
This is the joint density of a Dirichlet distribution $\textrm{Dir}(\theta_1,\dots,\theta_n)$.
\end{rem}

In Model III, we consider a continuous spectrum of possible types and we slightly modify the immigration process: when an individual arrives on the island, it starts a new population with an immigration rate proportional to its abundance on the metacommunity.
 More precisely, let $\Pi$ be a Poisson point process on $\R_+\times\R_+$ with intensity $dt\otimes xf(x)dx$ where $f$ is a nonnegative function such that $\theta:=\intpos xf(x)dx$ is finite. Then, write $\Pi:=((T_i,\Delta_i),i\geq 1)$ where  $T_1<T_2<\cdots$ are the times of a $\theta$-linear Poisson point process and $(\Delta_i,i\geq1)$ is a sequence of i.i.d random variables whose density is $\theta^{-1}xf(x)dx$ which is independent of $(T_i,i\geq1)$.
At time $T_i$, a new population starts out and it evolves like the continuous-time branching process with immigration defined for the two previous models with an immigration rate $\Delta_i$.The interpretation of this model is as follows:
for $x>0$, $f(x)$ represents the density of species with abundance $x$ in the metacommunity and at each immigration time, an individual of a species with abundance in $(x,x+dx)$ is chosen with probability $\frac{xf(x)}{\theta}dx$ proportional to its abundance in the metacommunity.

If $(Z^i(t),t\geq0)$ is the $i$-th oldest family,
$$Z^i(t)=I^i_{\Delta_i}(t-T_i)\1_{\{t\geq T_i\}},\quad t\geq0$$ where the $I^i_{\Delta_i}$'s are independent copies of $I_\Delta$, which, conditional on $\Delta$, evolves like the immigration process of the first two models with an immigration rate $\Delta$.
According to Theorem \ref{convpresquesureimmig}(ii), we know that $$e^{-\eta t}I_{\Delta}(t)\underset{t\rightarrow \infty}\longrightarrow G \textrm{ a.s.}$$ where conditional on $\Delta$, $G\sim\textrm{Gamma}(\Delta/b,c)$. We denote by $F$ its distribution tail $$F(v):=\p(G\geq v)=\intpos dx \frac{xf(x)}{\theta}\int_v^\infty \frac{e^{-ct}t^{x/b-1}c^{x/b}}{\Gamma(x/b)}dt.$$
Hence, we also have
\begin{prop}\label{noufr}
For $i\geq1$,
$$e^{-\eta t}Z^i(t)\underset{t\rightarrow \infty}\longrightarrow e^{-\eta T_i}G_i \textrm{ a.s.}
$$
where $(G_i,i\geq1)$ is a sequence of i.i.d. r.v. with the same distribution as $G$ and independent of $(T_i,i\geq1)$.
\end{prop}
We again denote by $I(t)$ the total population at time $t$:
$\displaystyle I(t):=\sum_{i\geq1}Z^i(t)$ and we obtain a result similar to Theorem \ref{convpresquesureimmig} concerning the asymptotic behavior of $I(t)$.

\begin{prop}\label{syco}If $\intpos x^2f(x)dx<\infty$ we have
$$e^{-\eta t}I(t)\overset{\textrm{a.s.}}{\underset{t\rightarrow \infty}\longrightarrow}\sum_{i\geq1}e^{-\eta T_i}G_i$$
and the Laplace transform of $\sigma:=\sum_{i\geq1}e^{-\eta T_i}G_i$ is
\begin{equation*}\label{Laplacesigma}\Esp\left[e^{-s\sigma}\right]=\exp\left(-\frac{\theta}{\eta}\intpos\frac{F(v)}{v}\left(1-e^{-sv}\right)dv\right).
\end{equation*}
Moreover, \begin{equation*}\Esp[\sigma]=\frac{1}{\eta bc}\intpos x^2f(x)dx<\infty.\end{equation*}
\end{prop}

We also have a result about abundances of different types
\begin{thm}[Model III]\label{abondmodele3}
We have
$$\left(\frac{Z^1(t)}{I(t)},\frac{Z^2(t)}{I(t)},\dots\right)\underset{{t\rightarrow\infty}}\longrightarrow\left(\frac{\sigma_1}{\sigma},\frac{\sigma_2}{\sigma},\dots\right) \textrm{ a.s.}$$
where $(\sigma_i,i\geq1)$ are the points of a non-homogeneous Poisson point process on $(0,\infty)$ with intensity measure $\frac{\theta}{\eta}\frac{F(y)}{y}dy$ and $\sigma=\sum_{i\geq1}\sigma_i$.
\end{thm}
The proofs of these two results will be done in Subsection \ref{model3}.
Notice that in this model, the limit only depends on the lifespan measure via the Malthusian parameter $\eta$ and the constant $c$.

\section{Proof of Proposition \ref{propCMJ}}\label{preuvepropCMJ}
\subsection{Some useful, technical lemmas}
Thereafter, we state some lemmas that will be useful in subsequent proofs.
\begin{lem}\label{supmoyenne}
Let $Y_1,Y_2,\dots$ be a sequence of i.i.d random variables with finite expectation. Then, if $S:=\sup_{n\geq1}\left(\frac{1}{n}\sum_{i=1}^{n}Y_i\right)$,
$$\Esp\left[(\log^+S)^k\right]<\infty,\quad k\geq1.$$
\end{lem}
\begin{proof}
According to Kallenberg \cite[p.184]{Kallenberg2002}, for $r>0$,
$$r\p(S\geq2r)\leq\Esp[Y_1;Y_1\geq r].$$
Hence, choosing $r=e^{s}/2$, we have for $s\geq0$,
$$\p(\lop S\geq s)\leq 2e^{-s}\Esp[Y_1]$$
and
$$\Esp[(\lop S)^k]=\intpos ks^{k-1}\p(\lop S\geq s)ds\leq2\Esp[Y_1]k\intpos s^{k-1}e^{-s}ds<\infty.$$
This completes the proof.
\end{proof}
\begin{lem}\label{suppoisson}
Let $A$ be a homogeneous Poisson process with parameter $\rho$. If $S:=\sup_{t>0}( A_t/t)$, then for $a>0$, $$\p(S>a)=\frac{\rho}{a}\vee1.$$ In particular, $$\forall k\geq1,\ \Esp\left[(\log^+S)^k\right]<\infty.$$
\end{lem}

\begin{proof}
If $a<\rho$, since $\lim_{t\rightarrow\infty}A_t/t=\rho$, $\p(S>a)=1$.
Let now $a$ be a real number greater than $\rho$. Then
$$\p(S\leq a)=\p(\forall t\geq0, at-A_t\geq0).$$
According to Bertoin \cite[chap.VII]{Levy_processes}, since $(at-A_t,t>0)$ is a L\'evy process with no positive jumps and with Laplace exponent $\phi(\lambda)=\lambda a-\rho(1-e^{-\lambda})$, we have as in (\ref{doubleexit})
$$\p(S\leq a)=\frac{H(0)}{H(\infty)}$$ where $H$ is the scale function associated with $(at-A_t,t>0)$ .
We compute $H(0)$ and $H(\infty)$ using Tauberian theorems \cite[p.10]{Levy_processes}:
\begin{itemize}
\item $\phi(\lambda)\underset{0}\sim\lambda(a-\rho)$ then $H(x)\underset{\infty}\sim (a-\rho)^{-1}$
\item $\phi(\lambda)\underset{\infty}\sim\lambda a$ then $H(x)\underset{0}\sim a^{-1}$
\end{itemize}
Hence, $$\p(S\leq a)=\frac{a-\rho}{a}=1-\frac{\rho}{a}, \quad a>\rho.$$
Then, {\setlength\arraycolsep{2pt}
\begin{eqnarray}\Esp[(\lop S)^k]&=&\intpos kr^{k-1}\p(\lop S\geq r)dr\nonumber\\
&\leq&\int_0^{\lop \rho}kr^{k-1}dr+\int_{\lop \rho}^\infty kr^{k-1}\frac{\rho}{e^{r}}dr<\infty\nonumber
\end{eqnarray}}
and the proof is completed.
\end{proof}
\subsection{Proof of Proposition \ref{propCMJ}(i)}
A. Lambert proved in \cite{Amaury_cours_Mexique} that
$\p(\textrm{Ext})=1-\eta/b$
and in \cite{Amaury_contour_splitting_trees} that, conditional on $\textrm{Ext}^c$,
$$e^{-\eta t}X(t)\overset{\mathcal{L}}{\underset{t\rightarrow\infty}\longrightarrow} E$$ where $E$ is an exponential random variable with parameter $c$. To obtain a.s. convergence, we use \cite[Thm 5.4]{Nerman_supercritical_CMJ} where O. Nerman  gives sufficient conditions for convergence of CMJ processes to hold almost surely.
Here, the two conditions of his theorem are satisfied. Indeed, the second one holds if there exists on $[0,\infty)$ an integrable, bounded, non-increasing positive function $h$ such that
 $$\Esp\left[\supt\left(\frac{e^{-\eta t}\1_{\{t<\zeta\}}}{h(t)}\right)\right]<\infty$$
 where we recall that $\zeta$ is the lifespan duration of a typical individual in the CMJ-process $X$. Then, choosing $h(t)=e^{-\eta t}$, this condition is trivially satisfied.

 The first one holds if there exists a non-increasing Lebesgue integrable positive function $g$ such that
\begin{equation}\label{condition CV ps}\int_0^\infty\frac{1}{g(t)}e^{-\eta t}\mu(dt)<\infty.\end{equation}
Taking  $g(t)=e^{-\beta t}$ with $\eta>\beta>0$ and recalling that $\mu(dt)=\int_{(t,\infty)}\!\Lambda(dr)dt$, we have
{\setlength\arraycolsep{2pt}
\begin{eqnarray}
\int_0^\infty\frac{1}{g(t)}e^{-\eta t}\mu(dt)&=&\intpos e^{(\beta-\eta)t}\int_t^\infty\Lambda(dr)dt=\int_{(0,\infty)}\!\Lambda(dr) \int_0^re^{(\beta-\eta)t}dt\nonumber\\
&\leq&C\int_{(0,\infty)} \!\!\Lambda(dr)=Cb<\infty\nonumber
\end{eqnarray}}
and condition (\ref{condition CV ps}) is fulfilled.
\subsection{Proof of Proposition \ref{propCMJ}(ii)}
We will say that a process $G$ satisfies condition (C) if
$$\Esp\left[\left(\log^+\sup_{t\geq0}e^{-\eta t}G_t\right)^2\right]<\infty$$
and our aim is to prove that, conditional on non-extinction, the homogeneous CMJ-process $(X(t),t\geq0)$ satisfies it.

According to Theorem 4.4.1.1 in  \cite{Amaury_cours_Mexique}, conditional on non-extinction of $(X(t),t\geq0)$,\begin{equation*}\label{decomposition}X(t)=X_t^{\infty}+X_t^{d}+X_t^g\end{equation*}
where \begin{itemize}
\item $X_t^\infty$ is the number of individuals alive at time $t$ and whose descendance is infinite. In particular, $(X_t^\infty,t\geq0)$ is a Yule process with rate $\eta$.
\item $X_t^d$ is the number of individuals alive at time $t$ descending from trees grafted on the right hand side of the Yule tree (right refers to the order of the contour of the planar splitting tree)
    $$X_t^d:=\sum_{i=1}^{\tilde N_t}\tilde X_i(t-\tilde T_i)$$
where
\begin{itemize}
\item $(\tilde X_i,i\geq1)$ is a sequence of i.i.d. splitting trees conditional on extinction and independent of $X^\infty$. We know that such trees have the same distribution as subcritical splitting trees with lifespan measure $\tilde\Lambda(dr)=e^{-\eta r}\Lambda(dr)$ (cf. \cite{Amaury_cours_Mexique}).
\item Conditionally on $(X_t^\infty,t\geq0)$, $(\tilde N_t,t\geq0)$ is an non-homogeneous Poisson process with mean measure $(b-\eta)X_t^\infty dt$ and independent of $(\tilde X_i)_i$. We denote its arrival times by $\tilde T_1,\tilde T_2,\dots$
\end{itemize}
\item $X_t^g$ is the number of individuals alive at time $t$ descending from trees grafted on the left hand side of the Yule tree (left also refers to the contour order).

    More specifically, let $(A,R)$ be a couple of random variables with joint law given by
    \begin{equation}\label{loiAetR}\p(A+R\in dz,R\in dr)=e^{-\eta r}dr\Lambda(dz),\ 0<r<z\end{equation}
    and let $((A_{i,j},R_{i,j}),i\geq0,j\geq1)$ be i.i.d random variables distributed as $(A,R)$.
    We consider the arrival times  $$T_{i,j}=\tau_i+A_{i,1}+A_{i,2}+\cdots+A_{i,j},\ i\geq0,j\geq1$$ where $0=\tau_0<\tau_1<\tau_2<\cdots$ are the splitting times of the Yule tree, that is, on each new infinite branch, we start a new $A$-renewal process independent of the others.
    We define for $t\geq0$,
    $$X_t^g:=\sum_{i,j}\hat X_{i,j}(t- T_{i,j})\1_{\{T_{i,j}\geq t\}}.$$
where $(\hat X_{i,j},i\geq0,j\geq1)$ is a sequence of i.i.d. splitting trees independent of $X^\infty$, conditioned on extinction and such that the unique ancestor of $\hat X_{i,j}$ has lifetime $R_{i,j}$.
We denote by $\hat N_t:=\#\{(i,j),\  T_{i,j}\leq t\}$ the number of graft times before $t$.
\end{itemize}
\begin{figure}[!ht]
\begin{center}
\includegraphics{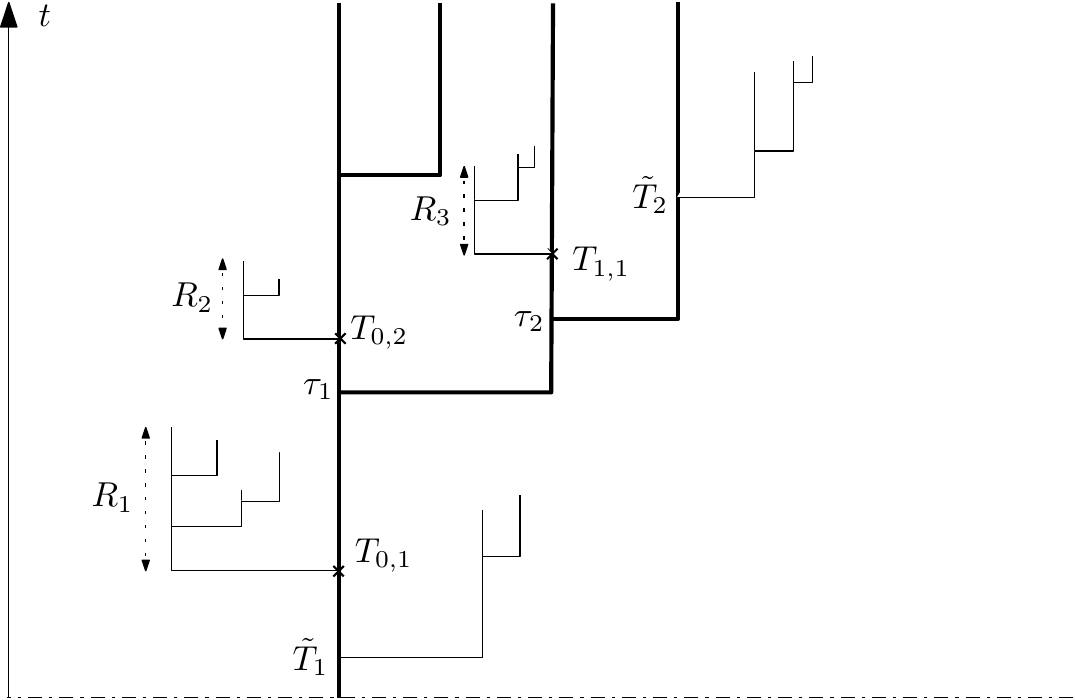}
\caption{Spine decomposition of a splitting tree. In bold, the Yule tree $X^\infty$ on which we graft on the left (at times $ T_{0,1},\dots$) the trees conditioned on extinction whose ancestors have lifetime durations distributed as $R$ and on the right (at time $\tilde T_1,\dots$) the trees conditioned on extinction.}
\label{figure}
\end{center}
\end{figure}

 In the following, we will prove that $(X_t,t\geq0)$ satisfies condition (C). To do this, using Minkowski inequality and the inequality
$$\forall x,y\geq0,\ \lop(x+y)\leq\lop x+\lop y+\log 2,$$  we only have to check that the three processes $X^\infty,X^g$ and $X^d$ satisfy (C).
\subsubsection{Proof of condition (C) for $X^\infty$}\quad\\
Since $(X_t^\infty,t\geq0)$ is a $\eta$-Yule process, $(e^{-\eta t}X_t^\infty,t\geq0)$ is a non-negative martingale \cite[p.111]{Athreya_Ney} and so by Doob's inequality \cite{Revuz1999},
\begin{equation}\label{kk}\Esp\left[\sup_{t\geq0}(e^{-\eta t}X^\infty_t)^2\right]\leq4\sup_{t\geq0}\Esp[(e^{-\eta t}X^\infty_t)^2]\end{equation}
Moreover, $\Esp[(X_t^\infty)^2]=2\left(e^{2\eta t}-e^{\eta t}\right)$ (again in \cite{Athreya_Ney}) and
$$\Esp[(e^{-\eta t}X^\infty_t)^2]=2e^{-2\eta t}\left(e^{2\eta t}-e^{\eta t}\right)\underset{t\rightarrow\infty}\longrightarrow2.$$
Hence, the supremum in the right hand side of (\ref{kk}) is finite and \begin{equation*}\label{ll}\Esp\left[\left(\sup_{t\geq0}\left(e^{-\eta t}X_t^\infty\right)\right)^2\right]<\infty.\end{equation*}
From now on, we will set $M:=\supt e^{-\eta t}X_t^\infty$. Since $\Esp[M^2]<\infty$, (C) is trivially satisfied  by $X^\infty$.

\subsubsection{Proof of condition (C) for $X^d$}\quad\\
We recall that $$X_t^d=\sum_{i=1}^{\tilde N_t}\tilde X_i(t-\tilde T_i).$$
Denote by $Y_i$ the total progeny of the conditioned splitting tree $\tilde X_i$, that is, the total number of descendants of the ancestor plus one. Then, a.s for all $t\geq0$ and $i\geq1$, we have $\tilde X_i(t-\tilde T_i)\leq Y_i$  and
\begin{equation*}\label{majpoptotale}X_t^d\leq\sum_{i=1}^{\tilde N_t}Y_i\textrm{ a.s.}\quad t\geq0.\end{equation*}
Hence, almost surely for all $t$,
$$e^{-\eta t}X_t^d\leq e^{-\eta t}\sum_{i=1}^{\tilde N_t}Y_i=\left(e^{-\eta t}\tilde N_t\right)\left(\frac{1}{\tilde N_t}\sum_{i=1}^{\tilde N_t}Y_i\right)$$
and, thanks to Minkowski's inequality,
{\setlength\arraycolsep{2pt}
\begin{eqnarray}
\Esp\left[\left(\log^+\sup_{t\geq0}\left(e^{-\eta t}X_t^d\right)\right)^2\right]^{1/2}&\leq&
\Esp\left[\left(\log^+\sup_{t\geq0}\left(e^{-\eta t}\tilde N_t\right)\right)^2\right]^{1/2}\nonumber\\
&&\qquad+\Esp\left[\left(\log^+\sup_{t>0}\left(\frac{1}{\tilde N_t}\sum_{i=1}^{\tilde N_t}Y_i\right)\right)^2\right]^{1/2}\label{pou}
\end{eqnarray}}
We first consider the second term in the right hand side of (\ref{pou}):
$$\Esp\left[\left(\log^+\sup_{t>0}\left(\frac{1}{\tilde N_t}\sum_{i=1}^{\tilde N_t}Y_i\right)\right)^2\right]
\leq \Esp\left[\left(\log^+\sup_{n\geq1}\left(\frac{1}{n}\sum_{i=1}^{n}Y_i\right)\right)^2\right]$$
since $(\tilde N_t,t\geq0)$ is integer-valued.
Thanks to Lemma \ref{supmoyenne}, this term is finite because $\Esp[Y_1]$ is finite. Indeed, $Y_1$ is the total progeny of a subcritical branching process and it is known \cite{Harris1963} that its mean is finite.\\

We are now interested in the first term in the r.h.s. of (\ref{pou}). We work conditionally on $X^\infty=:(f(t),t\geq0)$. Since we have $e^{-\eta t}\int_0^t f(s)ds\leq M$ using the supremum $M$ of $(e^{-\eta t}X_t^\infty,t\geq0)$,
$$e^{-\eta t}\tilde N_t=e^{-\eta t}\!\int_0^t\!f(s)ds\ \frac{\tilde N_t}{\int_0^tf(s)ds}\leq M\frac{\tilde N_t}{\int_0^tf(s)ds}$$
Moreover, $$\left(\tilde N_t,t\geq0\right)\overset{(d)}=\left(N'_{\int_0^tf(s)ds},t\geq0\right)$$ where $N'$ is a homogeneous Poisson process with parameter $b-\eta$.

Hence, using Minkowski's inequality,
{\setlength\arraycolsep{2pt}
\begin{eqnarray}
\Esp\left[\left(\lop \sup_{t\geq0}\left(e^{-\eta t}\tilde N_t\right)\right)^2\right]^{1/2}&\leq& \lop M+\Esp\left[\left(\lop \sup_{t>0}\left(\frac{N'_{\int_0^tf(s)ds}}{\int_0^tf(s)ds}\right)\right)^2\right]^{1/2}\nonumber\\
&=&\lop M+\Esp\left[\left(\lop \sup_{t>0}\left(\frac{N'_t}{t}\right)\right)^2\right]^{1/2}\nonumber
\end{eqnarray}}
 and the second term of the r.h.s. is finite using Lemma \ref{suppoisson}.

Hence, $(\tilde N_t,t\geq0)$ satisfies (C) since $\Esp[M^2]<\infty$ and $X^d$ as well, which ends this paragraph.

\subsubsection{Proof of condition (C) for $X^g$}\quad\\
We have $$X_t^g=\sum_{i=1}^{\hat N_t}\hat X_i(t-\hat T_i)$$
As in the previous section,
\begin{equation*}
e^{-\eta t}X_t^g\leq\left(e^{-\eta t}\hat N_t\right)\left(\frac{1}{\hat N_t}\sum_{i=1}^{\hat N_t}\hat Y_i\right) \textrm{ a.s.}
\end{equation*}
where $\hat Y_i$ is the total progeny of the conditioned CMJ-process $(\hat X_i(t),t\geq0)$.

Hence,
{\setlength\arraycolsep{2pt}
\begin{eqnarray}
\Esp\left[\left(\lop \supt\left(e^{-\eta t}X_t^g\right)\right)^2\right]^{1/2}&\leq&\Esp\left[\left(\lop \supt\left(e^{-\eta t}\hat N_t\right)\right)^2\right]^{1/2}\nonumber\\
&&\qquad+\Esp\left[\left(\lop \supn\left(\frac{1}{n}\sum_{i=1}^{n}\hat Y_i\right)\right)^2\right]^{1/2}\label{mr}
\end{eqnarray}}

We first prove that the second term in the r.h.s. is finite using Lemma \ref{supmoyenne}. We only have to check that $\Esp[\hat Y_1]$ is finite. We recall that $\hat Y_1$ is the total progeny of a splitting tree whose ancestor has random lifespan $R_1$ and conditioned on extinction. Conditioning on $R_1$, it is also the total progeny of a subcritical Bienaym\'e-Galton-Watson process starting from a Poisson random variable with mean $R_1$. Hence, $\Esp[\hat Y_1]$ is finite if $\Esp[R_1]$ is finite. As a consequence of (\ref{loiAetR}), we have that $\p(R_1\in dr)=e^{-\eta r}\int_r^{\infty}\Lambda(dz)dr$ and
{\setlength\arraycolsep{2pt}
\begin{eqnarray}
\Esp[R_1]&=&\int_{(0,\infty)} \Lambda(dz)\int_0^zre^{-\eta r}dr\nonumber\\
&=&-\int_{(0,\infty)} \Lambda(dz)\frac{ze^{-\eta z}}{\eta}+\int_{(0,\infty)} \Lambda(dz)\frac{1-e^{-\eta z}}{\eta^2}\nonumber\\
&=&\frac{\pse-1}{\eta}+\frac{1}{\eta}=\frac{\pse}{\eta}<\infty.\nonumber
\end{eqnarray}}

We are now interested in the first term of the r.h.s. of (\ref{mr}). We need to make calculations on $\hat N_t$ which is the total number of times of graftings $T_{i,j}$ less than or equal to $t$. Recall that for $i\geq0,j\geq1$, $T_{i,j}=\tau_i+\overline{A}_{i,j}$ where $\overline{A}_{i,j}:=A_{i,1}+\dots+A_{i,j}$ and that $\tau_i$ is the birth time of individual $i$ and that $\hat N_t$ is the sum of the numbers of graftings before $t$ on each of the $X_t^\infty$ branches.
 For $i\geq0$, denote by $\alpha_1^i,\alpha_2^i,\dots,$ the birth times of the daughters of individual $i$ and $\alpha^i_0=\tau_i$. For $k\geq1$, denote by $\tilde\tau_k^i:=\alpha^i_{k}-\alpha_{k-1}^i$ the interbirth times. In particular, $(\tilde\tau_k^i,i\geq0,k\geq1)$ are i.i.d. exponential r.v. with parameter $\eta$ since we consider a $\eta$-Yule tree.
\begin{figure}[!ht]
\begin{center}
\includegraphics{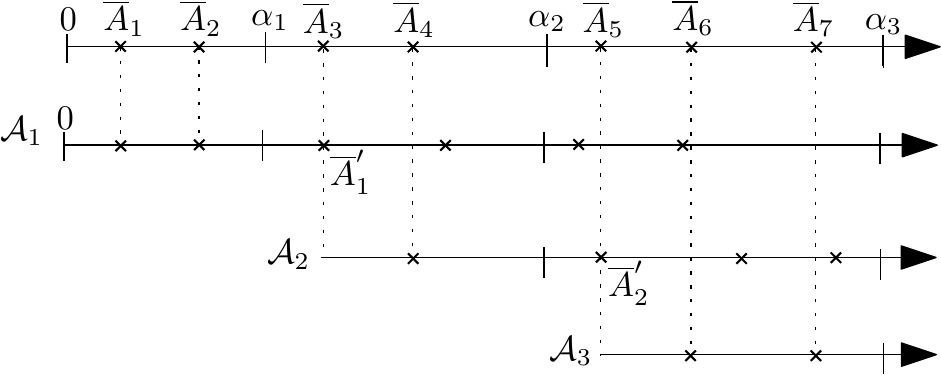}
\caption{Construction of the renewal process $(\overline{A}_j,j\geq1)$ by concatenation of the renewal processes $\mathcal A_k$.}
\label{concatenation}
\end{center}
\end{figure}

We enlarge the probability space by redefining the renewal processes $(\overline{A}_{i,j},j\geq1)$ from a doubly indexed sequence of i.i.d. $A$-renewal processes $(\mathcal{A}_{i,k},i\geq0,k\geq1)$.
    We define the process $(\overline{A}_{i,j},j\geq1)$ recursively by concatenation  of the $\mathcal A_{i,k}$'s as in Figure \ref{concatenation}. To simplify notation, we only define $(\overline{A}_{0,j},j\geq1)$ which will be denoted by $(\overline{A}_{j},j\geq1)$.
 Then, $$\overline{A}'_1:=\inf\{t>\alpha_1|\mathcal A_1\cap[t,+\infty)\neq\emptyset\},\quad C_1:=\#\mathcal A_1\cap[0,\alpha_1]+1,$$
and
$$\overline{A}_j:=\inf\{t\geq0|\#\mathcal A_1\cap[0,t]=j\},\quad j=1,\dots,C_1+1.$$
Moreover, for $l\geq1$, if one knows $C_l$ and $\overline{A}'_l$, let $r_l$ be the unique integer such that $\overline{A}'_l$ belongs to $]\alpha_{r_l},\alpha_{r_l+1}]$ and define
$$\overline{A}'_{l+1}:=\overline{A}'_{l}+\inf\{t>\alpha_{r_l+1}-\overline{A}'_l|\mathcal A_{r_l+1}\cap[t,+\infty)\neq\emptyset\},$$
$$C_{l+1}:=\#\mathcal A_{r_l+1}\cap[0,\alpha_{r_l+1}-\overline{A}'_l]+1,$$
and
$$\overline{A}_{C_1+\dots+C_l+j}:=\overline A'_l+\inf\{t\geq0|\#\mathcal A_{r_l+1}\cap[0,t]=j\},\quad j=1,\dots,C_{l+1}.$$
Then, $(\overline{A}_j,j\geq1)$ is a $A$-renewal process because we concatenated the independent renewal processes $(\mathcal A_k)$ stopped at the first renewal time after a given time. Indeed, one can see a renewal process as the range of a compound Poisson process whose jumps are distributed as $A_1$. The renewal process stopped at the first point after $t$ is then the range of a compound Poisson process stopped at the first hitting time $T$ of $[t,\infty)$, which is a stopping time. Then $(\overline{A}_j,j\geq1)$, as the range of the concatenation of independent compound Poisson processes killed at stopping times, is a compound Poisson process, by the strong Markov property. In conclusion, $(\overline{A}_j,j\geq1)$ is a $A$-renewal process.

According to previous computations, we have
$$\hat N_t\leq\sum_{i\geq0,l\geq1}C_l^i\1_{\{\alpha_{r_{l,i}}^i\leq t\}}$$ where $C_l^i$ and $r_{l,i}^i$ are the analogous notations as $C_l$ and $r_{l}$ for individual $i$.
Then if
$$D_k^i:=\#\mathcal{A}_{i,k}\cap[0,\tilde\tau_k^i]+1\qquad k\geq1,i\geq0,$$
since $\alpha_{r_{l,i}+1}^i-\overline A'_{l,i}\leq \alpha_{r_{l,i}+1}^i-\alpha_{r_{l,i}}^i=\tilde\tau_{r_{l,i}}^i$, then $C_l^i\leq D_{r_{l,i}}^i$ and
\begin{equation}\label{elole}\hat N_t\leq\sum_{i\geq0,l\geq1}D_{r_{l,i}}^i\1_{\{\alpha_{r_{l,i}}^i\leq t\}}\leq\sum_{i\geq0,k\geq1}D_k^i\1_{\{\alpha_k^i\leq t\}}\textrm{ a.s.}\end{equation}

Moreover, the random variables $(D^i_k,i\geq0,k\geq 1)$ are independent and identically distributed as 1 plus the value of a $A$-renewal process at an independent exponential time $E$ with parameter $\eta$
$$D:=\sup\{j\geq1,\ A_1+\cdots+A_j\leq E\}+1.$$
The sum of the r.h.s. of (\ref{elole}) has $2X^\infty_t-1$ terms. Indeed, each individual of the Yule tree contributes to $1+n_i$ terms in the sum where $n_i$ is the number of daughters of individual $i$ born before $t$. Then, there is $\displaystyle X_t^\infty+\sum_{i=0}^{X^\infty_t-1}n_i$ terms in the sum and $\sum_i n_i$ is the number of descendants of the individual 0 born before $t$ which equals $X^\infty_t-1$.

Hence, using Minkowski's inequality,
{\setlength\arraycolsep{2pt}\begin{eqnarray}\label{fr}
\Esp\left[\left(\lop \supt\left(e^{-\eta t}\hat N_t\right)\right)^2\right]^{1/2}&\leq&\Esp\left[\left(\lop \supt\left(e^{-\eta t}(2X_t^\infty-1)\right)\right)^2\right]^{1/2}\nonumber\\
&&\qquad\quad+\Esp\left[\left(\lop \supn\left(\frac{1}{n}\sum_{i=1}^n D_i\right)\right)^2\right]^{1/2}
\end{eqnarray}}
where $(D_i,i\geq1)$ are i.i.d. r.v. distributed as $D$.
The first term of (\ref{fr}) is smaller than $\Esp[(\lop(2M))^2]^{1/2}$ and the second term is finite by another use of Lemma \ref{supmoyenne} if $\Esp[ D_1]<\infty$. Denote by $(C_t,t\geq0)$ the renewal process whose arrival times are distributed as $A$. A little calculation from (\ref{loiAetR}) gives us that $\Esp[A]=\frac{m-1}{\eta}>0$. Thus, according to Theorem 2.3 of Chapter 5 in \cite{Durrett1999}, we have $$\lim_{t\rightarrow\infty}\frac{\Esp[C_t]}{t}=\frac{\eta}{m-1}$$ and so there exists $\kappa>0$ such that $\Esp[C_t]\leq\kappa t$ for $t\geq0$. Then,
$$\Esp[ D_1]=1+\Esp[C_E]=1+\intpos\eta e^{-\eta t}dt\Esp[C_t]\leq 1+\kappa \intpos\eta e^{-\eta t}dt=1+\frac{\kappa}{\eta}<\infty.$$
Finally, the r.h.s. of (\ref{fr}) and (\ref{mr}) are finite and $X^g$ satisfies condition (C).

\section{Proof of Theorem \ref{convpresquesureimmig}\label{preuvetheoprincipal}}
\subsection{Some preliminary lemmas}
We start with some properties about $W$ the scale function associated with $\psi$ and defined by (\ref{defW}).
\begin{lem}\label{props_de_W}
\begin{enumerate}[(i)]
\item $W(0)=1$,
\item $e^{-\eta t}W(t)\longrightarrow c^{-1}$ as $t\rightarrow\infty$,
\item $W$ is differentiable and $W\star\Lambda=bW-W'$ where $\star$ is convolution product.
\end{enumerate}
\end{lem}

\begin{proof}
\begin{enumerate}[(i)]
\item We have $$\intpos e^{-\lambda t} W(t)dt=\frac{1}{\psi(\lambda)}\underset{\lambda\rightarrow\infty}\sim\frac{1}{\lambda}$$
  because $\psi(\lambda)=\lambda-b+\int_{(0,\infty)} e^{-\lambda r}\Lambda(dr)$.
Then, by a Tauberian theorem \cite[p10]{Levy_processes}, $\lim_{t\rightarrow0}W(t)=1$.
\item (From \cite{Amaury_contour_splitting_trees}) For $\lambda>0$, using a Taylor expansion and $\psi(\eta)=0$, $$\psi(\lambda+\eta)\underset{\lambda\rightarrow0}\sim\lambda\psi'(\eta)=\lambda c.$$
Then $$\intpos W(t)e^{-\eta t}e^{-\lambda t}dt\underset{\lambda\rightarrow0}\sim\frac{1}{\lambda c}$$ and another Tauberian theorem entails that $W(t)e^{-\eta t}$ converges to $1/c$ as $t\rightarrow\infty$.
\item We first compute the Laplace transform of $W\star\Lambda$. Let $\lambda>\eta$
{\setlength\arraycolsep{2pt}
\begin{eqnarray}
\intpos e^{-\lambda t}W\star\Lambda(t)dt&=&\intpos e^{-\lambda t}W(t)dt\int_{(0,\infty)} e^{-\lambda r}\Lambda(dr)\nonumber\\
&=&\frac{1}{\psi(\lambda)}(\psi(\lambda)-\lambda+b)\nonumber
\end{eqnarray}}
Integrating by parts and using (i) and (ii),
$$\intpos e^{-\lambda t}W'(t)dt=\left[ e^{-\lambda t}W(t)\right]_0^\infty+\lambda\intpos e^{-\lambda t}W(t)dt=-1+\frac{\lambda}{\psi(\lambda)}$$
and so the Laplace transform of $bW-W'$ is $b/\psi(\lambda)+1-\lambda/\psi(\lambda)$ which equals that of $W\star \Lambda$. This completes the proof.
\end{enumerate}
\end{proof}

The following lemma deals with the convergence of random series:
\begin{lem}\label{series}
Let $(\zeta_i,i\geq1)$ be a sequence of i.i.d. positive random variables such that $\Esp[\lop \zeta_1 ]$ is finite and let $(\tau_i,i\geq1)$ be the arrival times of a Poisson point process with parameter $\rho$ independent of $(\zeta_i,i\geq1)$.
Then for any $r>0$, the series $\sum_{i\geq1}e^{-r \tau_i}\zeta_i$ converges a.s.
\end{lem}

\begin{proof} We have
\begin{equation}\label{expo}\sum_{i\geq1}e^{-r \tau_i} \zeta_i\leq\sum_{i\geq1}\exp\left(-i\left[r\frac{\tau_i}{i}-\frac{\lop \zeta_i}{i}\right]\right)\end{equation}
We use the following consequence of Borel-Cantelli's lemma: if $\xi_1,\xi_2,\dots$ are i.i.d. non-negative random variables, $$\limsup_{n\rightarrow \infty}\frac{\xi_n}{n}=0 \textrm{ or }\infty \textrm{ a.s.}$$
according to whether $\Esp[\xi_1]$ is finite or not. We use it with $\xi_i=\log^+ \zeta_i$. Hence, since $\Esp[\lop \zeta_1]$ is finite, $\lim_{i\rightarrow\infty}\lop\zeta_i/i=0$ a.s. Moreover, by the strong law of large numbers, $\tau_i/i$ converges almost surely to $1/\rho$ as $i$ goes to infinity. Then,
$$r\frac{\tau_i}{i}-\frac{\lop \zeta_i}{i}\underset{i\rightarrow\infty}\longrightarrow \frac{r}{\rho}>0 \textrm{ a.s.}$$
So the series in (\ref{expo}) converges a.s.
\end{proof}

\subsection{Proof of Theorem \ref{convpresquesureimmig}(i)}
In order to find the law of $I(t)$ the total population at time $t$, we use the fact that it is the sum of a Poissonian number of population sizes. More specifically, if we denote by $N_t$ the number of populations at time $t$, $(N_t,t\geq0)$ is a Poisson process with parameter $\theta$ and conditionally on $\{N_t=k\}$, the $k$-tuple $(T_1,\ldots,T_k)$ has the same distribution as $\left(U_{(1)},\ldots,U_{(k)}\right)$ which is the reordered $k$-tuple of $k$ independent uniform random variables on $[0,t]$.
Hence, conditionally on $\{N_t=k\}$,
{\setlength\arraycolsep{2pt}
\begin{eqnarray}I(t)&\overset{(d)}=&\sum_{i=1}^k X_i\left(t-U_{(i)}\right)\overset{(d)}=\sum_{i=1}^k X_i(t-U_{i})\overset{(d)}=\sum_{i=1}^k X_i(U_{i})\nonumber
\end{eqnarray}}
since all $U_{(i)}$'s appear in the sum and the $U_i$'s are independent from the $X_i$'s.
Hence, conditionally on $\{N_t=k\}$, $I(t)$ has the same distribution as a sum of $k$ i.i.d. r.v. with law $X(U)$.
Then, {\setlength\arraycolsep{2pt}
\begin{eqnarray}
G_t(s)&=&\sum_{k\geq0}\Esp\left[\left.s^{I(t)}\right|N_t=k\right]\p(N_t=k)\nonumber\\
&=&\sum_{k\geq0}\Esp\left[s^{X_1(U_1)}\right]^k\frac{(t\theta)^k}{k!}e^{-\theta t}\nonumber\\
&=&e^{-\theta t}\exp\left(t\theta\Esp\left[s^{X_1(U_1)}\right]\label{equa1}\right)
\end{eqnarray}}

We now compute the law of $X(t)$ for $t>0$ and then we will compute the law of $X_1(U_1)$.
Using (\ref{prop1loideX(t)}), (\ref{prop2loideX(t)}) and Lemma \ref{props_de_W}(iii), we have
{\setlength\arraycolsep{2pt}
\begin{eqnarray}\p(X(t)=0)&=&\int_{(0,\infty)} \tilde\p_r(X(t)=0)\frac{\Lambda(dr)}{b}=\int_{(0,\infty)}\frac{W(t-x)}{W(t)}\frac{\Lambda(dr)}{b}\nonumber\\
&=&\frac{1}{bW(t)}W\star\Lambda(t)=1-\frac{W'(t)}{bW(t)}\nonumber
\end{eqnarray}}
and for $n\in\Nat^*$,
{\setlength\arraycolsep{2pt}
\begin{eqnarray}
\p(X(t)=n)&=&\int_{(0,\infty)} \tilde\p_r(X(t)=n)\frac{\Lambda(dr)}{b}\nonumber\\
&=&\frac{1}{bW(t)}\left(1-\frac{1}{W(t)}\right)^{n-1}\Big(b-W(t)^{-1}W\star\Lambda(t)\Big)\nonumber\\
&=&\left(1-\frac{1}{W(t)}\right)^{n-1}\frac{W'(t)}{bW(t)^2}.\nonumber
\end{eqnarray}}
We now compute $\p(X_1(U_1)=n)$
{\setlength\arraycolsep{2pt}
\begin{eqnarray}
\p(X_1(U_1)=0)&=&\frac{1}{t}\int_0^t\p(X_1(u)=0)du=1-\frac{1}{tb}\int_0^t\frac{W'(u)}{W(u)}du\nonumber\\
&=&1-\frac{\log W(t)}{tb}\nonumber
\end{eqnarray}}
because $W(0)=1$. For $n>0$,
{\setlength\arraycolsep{2pt}
\begin{eqnarray}
\p(X_1(U_1)=n)&=&\frac{1}{t}\int_0^t\p(X_1(u)=n)du=\frac{1}{t}\int_0^t\left(1-\frac{1}{W(u)}\right)^{n-1}\frac{W'(u)}{bW(u)^2}du\nonumber\\
&=&\frac{1}{bt}\int_{W(t)^{-1}}^1\frac{(1-u)^{n-1}}{n}du=\frac{\left(1-1/W(t)\right)^n}{btn}.\nonumber
\end{eqnarray}}

We are now able to compute the generating function of $X_1(U_1)$. For $s,t>0$,
{\setlength\arraycolsep{2pt}
\begin{eqnarray}
\Esp\left[s^{X_1(U_1)}\right]&=&\frac{1}{bt}\sum_{n\geq1}\frac{s^n}{n}\left(1-\frac{1}{W(t)}\right)^n+1-\frac{\log W(t)}{tb}\nonumber\\
&=&1-\frac{1}{bt}\left[\log\Big(1-s(1-1/W(t))\Big)+\log W(t)\right]\nonumber\\
&=&1-\frac{1}{bt}\log\Big(W(t)+s(1-W(t))\Big).\nonumber
\end{eqnarray}}
Finally for $t,s>0$, according to (\ref{equa1}),
$$G_t(s)=e^{-\theta t}\exp\left(t\theta\Esp\left[s^{X_1(U_1)}\right]\right)=\left(W(t)+s(1-W(t))\right)^{-\theta/b}.$$
which is the p.g.f. of a negative binomial distribution with parameters $1-W(t)^{-1}$ and $\theta/b$.

\subsection{Proof of Theorem \ref{convpresquesureimmig}(ii)}
We first prove the almost sure convergence. Splitting $I(t)$ between the surviving and the non-surviving populations, we have
\begin{equation}\label{splitting}e^{-\eta t}I(t)=\sum_{i\geq1}e^{-\eta t}Z^{(i)}(t)+\sum_{i\geq1}e^{-\eta t}X_{i}(t-T_{i})\1_{\{t\geq T_{i}\}\cap\ext_i}\end{equation}
where for $i\geq1$, $\ext_i$ denotes the extinction of the process $X_i$.
We will show that for each of these two terms, we can exchange summation and limit, so that in particular, the second term vanishes as $t\rightarrow\infty$.

We first treat the second term of the r.h.s. of (\ref{splitting}). We have
$$C_t:=\sum_{i\geq1}X_{i}(t-T_{i})\1_{\{t\geq T_{i}\}\cap\ext_i}\leq\sum_{i\geq1}\1_{\{t\geq T_i\}}Y_i\1_{\ext_i} \textrm{ a.s.}$$
where $Y_i$ is the total progeny of the $i$-th population which does not survive.
Moreover, $\Esp\left[Y_1\1_{\ext_1}\right]\leq\Esp[Y_1]$, $Y_1$ is the total progeny of a subcritical Bienaym\'e-Galton-Watson process so its mean is finite. Hence, since the r.h.s. in the previous equation is a compound Poisson process with finite mean, it grows linearly and $e^{-\eta t}C_t$ vanishes as $t\rightarrow\infty$.

 To exchange summation and limit in the first term of the r.h.s. of (\ref{splitting}), we will use the dominated convergence theorem. By Proposition \ref{conv pop}, we already know that $e^{-\eta t}Z^{(i)}(t)$ a.s. converges as $t$ goes to infinity to $e^{-\eta T^{(i)}}E_i$. Hence, it is sufficient to prove that
\begin{equation}\label{mmm}\sum_{i\geq1}\sup_{t\geq 0}\left(e^{-\eta t}Z^{(i)}(t)\right)<\infty \textrm{ a.s.}\end{equation}
Since $$\sup_{t\geq 0}\left(e^{-\eta t}Z^{(i)}(t)\right)=e^{-\eta T^{(i)}}\sup_{t\geq0}\left(e^{-\eta t}X_{(i)}(t)\right),$$
we have
{\setlength\arraycolsep{2pt}
\begin{eqnarray}
\sum_{i\geq1}\sup_{t\geq 0}\left(e^{-\eta t}Z^{(i)}(t)\right)&=&\sum_{i\geq1}e^{-\eta T^{(i)}}J_i\nonumber
\end{eqnarray}}
where $J_i:=\sup_{t\geq0}\left(e^{-\eta t}X_{(i)}(t)\right)$ for $i\geq1$ and $J_1,J_2,\dots$ are i.i.d.
Thus, using Lemmas \ref{series}, this series a.s. converges
if $\Esp[\lop J_1]$ is finite, which is checked thanks to Proposition \ref{propCMJ}(ii).

Then we get (\ref{mmm}) and using the dominated convergence theorem,
$$I:=\lim_{t\rightarrow\infty}e^{-\eta t}I(t)=\sum_{i\geq1}e^{-\eta T^{(i)}}E_i  \textrm{ a.s.}$$

In order to find the law of $I$, we compute its Laplace transform. For $a>0$, using part(i) of this theorem, $$\Esp\left[e^{-ae^{-\eta t}I(t)}\right]=G_t(e^{-ae^{-\eta t}})=\left(e^{-ae^{-\eta t}}+\left(1-e^{-ae^{-\eta t}}\right)W(t)\right)^{-\theta/b}$$
and
$$e^{-ae^{-\eta t}}+\left(1-e^{-ae^{-\eta t}}\right)W(t)\underset{t\rightarrow\infty}\sim 1+ae^{-\eta t}W(t)\underset{t\rightarrow\infty}\longrightarrow 1+\frac{a}{c}$$ using Lemma \ref{props_de_W} (ii).
Then,
$$\Esp\left[e^{-aI}\right]=\left(\frac{c}{a+c}\right)^{\theta/b}$$
which is the Laplace transform of a $\textrm{Gamma}(\theta/b,c)$ random variable.

\section{Other proofs}\label{other_proofs}
\subsection{Proof for Model I}\label{bloup}
To prove Theorem \ref{thm principal splitting trees}, we will follow Tavar\'e's proof \cite{Tavaré_BirthProcImm}.
We begin with a technical lemma which will be useful in the proof of this theorem and in forthcoming proofs.
\begin{lem}\label{ppp}
Let $(T_i,i\geq1)$ be the arrival times of a Poisson process with parameter $\rho$ and $\zeta_1,\zeta_2,\dots$ be i.i.d. r.v. independent of the $T_i$'s. Denote by $g$ the density of $\zeta_1$ with respect to Lebesgue measure and by $F(v):=\p(\zeta_1\geq v)$ its distribution tail.  Then for $r>0$, $(e^{-r T_i}\zeta_i,i\geq1)$ are the points of an non-homogeneous Poisson point process on $(0,\infty)$ with intensity measure $\frac{\rho}{r}\frac{F(v)}{v}dv.$
\end{lem}
 \begin{proof}We first study the collection $\Pi=\left\{(T_{i},\zeta_i), i\geq1\right\}$. This is a Poisson point process on $(0,\infty)\times(0,\infty)$ with intensity measure $\rho g(y)dtdy$. Then, $(e^{-r T_i}\zeta_i,i\geq1)$ is a Poisson point process whose intensity measure is the image of $\rho g(y)dtdy$ by $(t,y)\mapsto e^{-r t}y$. We now compute it. Let $h$ be a non-negative mapping. Changing variables, we get
{\setlength\arraycolsep{2pt}
\begin{eqnarray}
\intpos\intpos h(e^{-rt}y)\rho g(y)dtdy&=&\frac{\rho}{r}\intpos h(v)dv\int_0^1 g\left(\frac{v}{u}\right)\frac{du}{u^2}\nonumber\\
&=&\frac{\rho}{r}\intpos h(v)\frac{F(v)}{v}dv\nonumber
\end{eqnarray}}
and the proof is completed.
\end{proof}

We are now able to prove Theorem \ref{thm principal splitting trees}. By Proposition \ref{conv pop} and Theorem \ref{convpresquesureimmig},
$$I(t)^{-1}(Z^{(1)}(t),Z^{(2)}(t),\dots)=\frac{e^{-\eta t}(Z^{(1)}(t),Z^{(2)}(t),\dots)}{e^{-\eta t}I(t)}\underset{t\rightarrow\infty}\longrightarrow \left(\frac{\sigma_1}{\sigma},\frac{\sigma_2}{\sigma},\dots\right)\quad\textrm{a.s.}$$
where $\sigma_i:=\exp\left(-\eta T^{(i)}\right)E_i$ and $\sigma:=\sum_{i\geq1}\sigma_i$.

Moreover, the $(\sigma_i)_{i\geq 1}$ are the points of a non-homogeneous Poisson point process on $(0,\infty)$ with intensity measure $\displaystyle\frac{\theta}{b}\frac{e^{-c  y}}{y}dy$ thanks to Lemma \ref{ppp} with $\rho=\theta\eta/b$, $r=\eta$ and $F(v)=e^{-cv}$ because $(T^{(i)})_{i\geq 1}$ is a Poisson process of rate $\theta\eta/b$ and $(E_i)_{i\geq1}$ is an independent sequence of i.i.d. exponential variables with parameter $c$.

   According to \cite[p. 89]{Bertoin2006}, the Poisson point process $(\sigma_i)_{i\geq 1}$ satisfies $\sigma=\sum_{i\geq1}\sigma_i<\infty$ (actually $\sigma$ has a Gamma distribution) and the vector $$\left(\frac{\sigma_1}{\sigma},\frac{\sigma_2}{\sigma},\dots\right)$$ follows the GEM distribution with parameter $\theta/b$ and is independent of $\sigma$.

\subsection{Proof for Model II}\label{bloup2}
We will prove Theorem \ref{thm multitype}.
We recall that in Model II, immigrants are of type $i$ with probability $p_i$.
Denote by $N^i(t)$ the number of immigrants of type $i$ which arrived before time $t$. Then, $(N^i(t),t\geq0)$ is a Poisson process with parameter $\theta p_i$ and the processes $(N^i,i\geq 1)$ are independent. Hence, $I_1(t),I_2(t),\dots$ are independent and their asymptotic behaviors are the same as $I(t)$ in Theorem \ref{convpresquesureimmig}, replacing $\theta$ with $\theta p_i$.
Then we have
$$e^{-\eta t}I_i(t)\underset{t\rightarrow\infty}\longrightarrow I_i\quad \textrm{ a.s.}\qquad i\geq 1$$
where the $I_i$'s are independent and $I_i$ has a Gamma distribution $\Gamma(\theta_i,c)$ (recall that $\theta_i=\theta p_i/b$).

Moreover, $$e^{-\eta t}I(t)\underset{t\rightarrow\infty}\longrightarrow I\quad \textrm{ a.s.}$$
where $I\sim\Gamma(\theta/b,c)$.
Therefore, for
$r\geq1$,$$\lim_{t\rightarrow\infty}I(t)^{-1}(I_1(t),\dots,I_r(t))=\left(\frac{I_1}{I},\dots,\frac{I_r}{I}\right)\textrm{\quad
a.s.}$$
In order to investigate the law of this $r$-tuple, we prove
that
\begin{equation}\label{sansnom}I=\sum_{i\geq1}I_i \quad \textrm{
a.s.}\end{equation}
First, by Fatou's lemma,
$$\liminf_{t\rightarrow\infty}e^{-\eta t}\sum_{i\geq1} I_i(t)\geq
\sum_{i\geq1}\liminf_{t\rightarrow\infty} e^{-\eta t}I_{i}(t)\textrm{ a.s.}$$
and so $$I\geq\sum_{i\geq1}I_i \textrm{ a.s.}$$
Second,$$\Esp\left[\sum_{i\geq1}I_i\right]=\sum_{i\geq1}\frac{\theta_i}{c}=\frac{\theta
}{bc}=\Esp[I].$$
The last two equations yield (\ref{sansnom}).
For $1\leq i\leq r$, we can write
$$\frac{I_i}{I}=\frac{I_i}{I_1+\cdots+I_r+I^*}$$ where $I^*$ is
independent of $(I_i,1\leq i\leq r)$ and has a Gamma
distribution $\Gamma(\theta/b-\overline{\theta}_r,c)$ with $\overline{\theta}_r:=\sum_{i=1}^r\theta_i$.
Hence, one can compute the joint density of the $r$-tuple $\left(I_1/I,\dots,I_r/I\right)$ as follows
$$f(x_1,\dots,x_r)=\frac{\Gamma(\theta/b)}{\Gamma(\theta/b-\overline{\theta}_r)\prod_{i=1}^r\Gamma(\theta_i)}x_1^{\theta_1-1}\!\!\cdots x_r^{\theta_r-1}(1-x_1-\cdots-x_r)^{\theta/b-\overline{\theta}_r-1}$$
for $x_1,\dots,x_r>0$ satisfying $x_1+\cdots+x_r<1$.
This joint density is exactly that of $(P'_1,\dots,P'_r)$ defined in the statement of the theorem.

\subsection{Proofs for Model III}\label{model3}
We first prove the almost sure convergence in Proposition \ref{syco}. In order to do that, we use the same arguments as in the proof of Theorem \ref{convpresquesureimmig}(ii): we will use the dominated convergence theorem for the sum
$$e^{-\eta t}I(t)=\sum_{i\geq1}e^{-\eta t}I_{\Delta_i}^i(t-T_i)\1_{\{t\geq T_i\}}.$$
As in a previous proof, this sum is bounded by $\displaystyle\sum_{i\geq1}e^{-\eta T_i}\supt\left(e^{-\eta t}I_{\Delta_i}^i(t)\right)$ which, according to Lemma \ref{series}, is a.s. finite if $$\Esp\left[\lop\supt\left(e^{-\eta t}I_{\Delta}^1(t)\right)\right]<\infty.$$
However, $\displaystyle I_{\Delta}^1(t)=\sum_{i\geq1}X^i(t-\tilde T_i)\1_{\{t\geq\tilde T_i\}}$ where conditionally on $\Delta$, $(\tilde T_i,i\geq1)$ is a Poisson process with parameter $\Delta$.
Hence, $$\supt\left(e^{-\eta t}I_\Delta(t)\right)\leq\sum_{i\geq1}e^{-\eta\tilde T_i}\supt\left(e^{-\eta t}X^i(t)\right)=\sum_{i\geq1}e^{-\eta\tilde T_i}J_i$$ where $J_1,J_2,\dots$ is an i.i.d. sequence of random variables independent from $\tilde T_1,\tilde T_2,\dots$ distributed as $\supt( e^{-\eta t}X(t))$ where $(X(t),t\geq0)$ is a homogeneous CMJ-process.
According to Proposition \ref{propCMJ}(ii), we know that $\Esp[(\lop J_1)^2]<\infty$.

We define for $i\geq1$, $\varsigma_i:=e^{-\eta\tilde T_i}J_i$ and $\vs:=\sum_{i\geq1}\vs_i$ and we have to prove that $\Esp[\lop\vs]$ is finite.
To do that, we first work conditionally on $\Delta$. According to Lemma \ref{ppp}, we know that $(\vs_i,i\geq1)$ are the points of a non-homogeneous Poisson process on $(0,\infty)$ with intensity measure $\frac{\Delta}{\eta}\frac{L(v)}{v}dv$ where $L(v):=\p(J\geq v)$.
Then, using the inequality $$\lop(x+y)\leq\lop x+\lop y +\log2,\quad x,y\geq0,$$
we have
\begin{equation}\label{decompositionsigma}\Esp[\lop\vs]\leq \log2+\Esp\left[\lop\sum_{i\geq1}\vs_i\1_{\{\vs_i\leq1\}}\right]+\Esp\left[\lop\sum_{i\geq1}\vs_i\1_{\{\vs_i>1\}}\right]\end{equation}
We first consider the second term of the r.h.s.
{\setlength\arraycolsep{2pt}
\begin{eqnarray}
\Esp\left[\lop\sum_{i\geq1}\vs_i\1_{\{\vs_i\leq1\}}\right]&\leq&\Esp\left[\sum_{i\geq1}\vs_i\1_{\{\vs_i\leq1\}}\right]=\int_0^1v\frac{\Delta}{\eta}\frac{L(v)}{v}dv\leq\frac{\Delta}{\eta}
.\label{petitssigma}\end{eqnarray}}
Then, we compute the third term of the r.h.s. of (\ref{decompositionsigma}): if $A:=\sup_i\vs_i$,
{\setlength\arraycolsep{2pt}
\begin{eqnarray}
\Esp\left[\lop\sum_{i\geq1}\vs_i\1_{\{\vs_i>1\}}\right]&\leq&\Esp\left[\lop\left(A\cdot\#\{i\geq1|\vs_i>1\}\right)\right]\nonumber\\
&\leq&\Esp\left[\lop A\right]+\Esp\left[\lop\#\{i\geq1|\vs_i>1\}\right]\nonumber
\end{eqnarray}}
 Furthermore, the number of $\vs_i$ greater than 1 has a Poisson distribution with parameter $\int_1^\infty\frac{\Delta}{\eta}\frac{L(v)}{v}dv$. Since $\Esp[\lop J]<\infty$,

$$\intpos \p(\lop J\geq s)ds=\intpos \p(J\geq e^s)ds=\int_1^\infty\frac{L(v)}{v}dv<\infty.$$
Then,
\begin{equation}\label{cardinalsigma}\Esp\left[\lop\#\{i\geq1|\vs_i>1\}\right]\leq\Esp\left[\#\{i\geq1|\vs_i>1\}\right]=\frac{\Delta}{\eta}\int_1^\infty\frac{L(v)}{v}dv\leq C\Delta\end{equation}
where $C$ is a finite constant which does not depend on $\Delta$.

We now want to study $A$: $$\p(A\leq x)=\p(\#\{i\geq1|\vs_i>x\}=0)=\exp\left(-\int_x^\infty\frac{\Delta}{\eta}\frac{L(v)}{v}dv\right),\quad x>0.$$
So that
$$\p(A\in dx)=\frac{\Delta}{\eta}\frac{L(x)}{x}\exp\left(-\int_x^\infty\frac{\Delta}{\eta}\frac{L(v)}{v}dv\right)dx.$$
Then, {\setlength\arraycolsep{2pt}
\begin{eqnarray}\Esp[\lop A]&=&\int_1^\infty\log x\frac{\Delta}{\eta}\frac{L(x)}{x}\exp\left(-\int_x^\infty\frac{\Delta}{\eta}\frac{L(v)}{v}dv\right)dx\nonumber\\
&\leq& \frac{\Delta}{\eta}\int_1^\infty\log x\frac{L(x)}{x}dx=\frac{\Delta}{\eta}\intpos uL(e^u)du\nonumber\\
&\leq&\frac{\Delta}{\eta}\intpos u\p(\lop J\geq u)du\leq C'\Delta\label{majoration max}
\end{eqnarray}}
where $C'$ is a finite constant since $\Esp\left[\left(\lop J\right)^2\right]<\infty$ according to Proposition \ref{propCMJ}(ii).
Hence, with (\ref{petitssigma}), (\ref{cardinalsigma}) and (\ref{majoration max}) we have

$$\Esp[\lop \vs|\Delta]\leq \log2+C''\Delta.$$
Then, $\Esp[\lop \vs]$ is finite because $\Esp[\Delta]=\theta^{-1}\intpos x^2 f(x)dx<\infty$.
and, using the dominated convergence theorem, $e^{-\eta t}I(t)$ a.s. converges toward $\sigma=\sum_{i\geq1}e^{-\eta T_i}G_i$ as $t\rightarrow\infty$.\\

We now compute the law of the limit $\sigma$. We define $\sigma_i:=\exp\left(-\eta T_{i}\right)G_i$ for $i\geq 1$.
Then, using Lemma \ref{ppp}, $(\sigma_i)_{i\geq 1}$ are the points of a non-homogeneous Poisson point process on $(0,\infty)$ with intensity measure $\frac{\theta}{\eta}\frac{F(y)}{y}dy$ where $F(y)=\p(G\geq y)$.
To compute the Laplace transform of $\sigma$, we use the exponential formula for Poisson processes: for $s >0$,
$$\Esp\left[e^{-s\sigma}\right]=\exp\left(-\frac{\theta}{\eta}\intpos\frac{F(v)}{v}\left(1-e^{-sv}\right)dv\right).$$
and to get the expectation of $\sigma$, we differentiate the last displayed equation at $0$:
$$\Esp[\sigma]=\frac{\theta}{\eta}\intpos F(v)dv=\frac{\theta}{\eta}\Esp[G]$$
and
{\setlength\arraycolsep{2pt}
\begin{eqnarray}
E[G]&=&\theta^{-1}\intpos xf(x)\frac{x}{b}\frac{1}{c}dx<\infty\label{esp(sigma)}\nonumber
\end{eqnarray}}
Hence, $$\Esp[\sigma]=\frac{1}{\eta bc}\intpos x^2f(x)dx<\infty$$ which ends the proof of Proposition \ref{syco}.

 It remains to prove Theorem \ref{abondmodele3} that is to show that the vector $(Z^1(t),Z^2(t),\dots)/I(t)$ a.s. converges to a Poisson point process with intensity measure $\frac{\theta}{\eta}\frac{F(y)}{y}dy$. It  is straightforward using previous calculations, Propositions \ref{noufr} and \ref{syco}.

\begin{rem}
Thanks to similar calculations as in Theorem \ref{convpresquesureimmig}(i), we can compute the generating function of $I(t)$  \begin{equation*}\label{cc}\Esp\left[s^{I(t)}\right]=\exp\left(-\int_0^t du\left(\theta-\intpos\frac{xf(x)}{\left(W(u)(1-s)+s\right)^{x/b}}dx\right)\right),\ s\in[0,1].\end{equation*}
and we can deduce the law of $\sigma$  in another way.

\end{rem}
\section*{Acknowledgments}
I want to thank my supervisor, Amaury Lambert,
for his very helpful remarks. My thanks also to the referee for his careful check and advice.

\bibliographystyle{apt}
\bibliography{ma_biblio}

\begin{thebibliography}{10}

\bibitem{Athreya_Ney}
{\sc Athreya, K.~B. and Ney, P.~E.} (1972).
\newblock {\em Branching processes}.
\newblock Springer-Verlag, New York.
\newblock Die Grundlehren der mathematischen Wissenschaften, Band 196.

\bibitem{Levy_processes}
{\sc Bertoin, J.} (1996).
\newblock {\em L\'evy processes} vol.~121 of {\em Cambridge Tracts in
  Mathematics}.
\newblock Cambridge University Press, Cambridge.

\bibitem{Bertoin2006}
{\sc Bertoin, J.} (2006).
\newblock {\em Random fragmentation and coagulation processes} vol.~102 of {\em
  Cambridge Studies in Advanced Mathematics}.
\newblock Cambridge University Press, Cambridge.

\bibitem{Caswell1976}
{\sc Caswell, H.} (1976).
\newblock Community structure: A neutral model analysis.
\newblock {\em Ecological Monographs\/} {\bf 46,} 327--354.

\bibitem{Donnelly1986}
{\sc Donnelly, P. and Tavar{\'e}, S.} (1986).
\newblock The ages of alleles and a coalescent.
\newblock {\em Adv. in Appl. Probab.\/} {\bf 18,} 1--19.

\bibitem{Durrett1999}
{\sc Durrett, R.} (1999).
\newblock {\em Essentials of stochastic processes}.
\newblock Springer Texts in Statistics. Springer-Verlag, New York.

\bibitem{Ethier1990}
{\sc Ethier, S.~N.} (1990).
\newblock The distribution of the frequencies of age-ordered alleles in a
  diffusion model.
\newblock {\em Adv. in Appl. Probab.\/} {\bf 22,} 519--532.

\bibitem{Fisher1943}
{\sc Fisher, R., Corbet, A. and Williams, C.} (1943).
\newblock The relation between the number of species and the number of
  individuals in a random sample of an animal population.
\newblock {\em Journal of Animal Ecology\/} {\bf 12,} 42--58.

\bibitem{Geiger1997}
{\sc Geiger, J. and Kersting, G.} (1997).
\newblock Depth-first search of random trees, and {P}oisson point processes.
\newblock In {\em Classical and modern branching processes ({M}inneapolis,
  {MN}, 1994)}.
\newblock vol.~84 of {\em IMA Vol. Math. Appl.} Springer, New York
  pp.~111--126.

\bibitem{Harris1963}
{\sc Harris, T.~E.} (1963).
\newblock {\em The theory of branching processes}.
\newblock Die Grundlehren der Mathematischen Wissenschaften, Bd. 119.
  Springer-Verlag, Berlin.

\bibitem{Hubbell}
{\sc Hubbell, S.} (2001).
\newblock {\em The Unified Neutral Theory of Biodiversity and Biogeography}.
\newblock Princeton Univ. Press, NJ.

\bibitem{Jagers_BP_with_bio}
{\sc Jagers, P.} (1975).
\newblock {\em Branching processes with biological applications}.
\newblock Wiley-Interscience, London.

\bibitem{Kallenberg2002}
{\sc Kallenberg, O.} (2002).
\newblock {\em Foundations of modern probability} second~ed.
\newblock Probability and its Applications (New York). Springer-Verlag, New
  York.

\bibitem{Kar_McGreg67}
{\sc Karlin, S. and McGregor, J.} (1967).
\newblock The number of mutant forms maintained in a population.
\newblock {\em Proc. 5th Berkeley Symposium Math. Statist. Prob.\/} {\bf IV,}
  415--438.

\bibitem{Kendall1948}
{\sc Kendall, D.~G.} (1948).
\newblock On some modes of population growth leading to {R}. {A}. {F}isher's
  logarithmic series distribution.
\newblock {\em Biometrika\/} {\bf 35,} 6--15.

\bibitem{Amaury_cours_Mexique}
{\sc Lambert, A.} (2008).
\newblock Population dynamics and random genealogies.
\newblock {\em Stoch. Models\/} {\bf 24,} 45--163.

\bibitem{Amaury_contour_splitting_trees}
{\sc Lambert, A.} (2010).
\newblock The contour of splitting trees is a \textsc{L}\'{e}vy process.
\newblock {\em Ann. Probab.\/} {\bf 38,} 348--395.

\bibitem{Amaury-Immig-Mut}
{\sc Lambert, A.} (2010).
\newblock Species abundance distributions in neutral models with immigration or
  mutation and general lifetimes.
\newblock Preprint available at
  \url{http://www.proba.jussieu.fr/pageperso/amaury/index_fichiers/MigraMut.pd%
f}. To appear in \emph{J. Math. Biol.}

\bibitem{McArthur_Wilson}
{\sc MacArthur, R. and Wilson, E.} (1967).
\newblock {\em The Theory of Island Biogeography}.
\newblock Princeton Univ. Press, NJ.

\bibitem{Nerman_supercritical_CMJ}
{\sc Nerman, O.} (1981).
\newblock On the convergence of supercritical general ({C}-{M}-{J}) branching
  processes.
\newblock {\em Z. Wahrsch. Verw. Gebiete\/} {\bf 57,} 365--395.

\bibitem{Revuz1999}
{\sc Revuz, D. and Yor, M.} (1999).
\newblock {\em Continuous martingales and {B}rownian motion} third~ed. vol.~293
  of {\em Grundlehren der Mathematischen Wissenschaften [Fundamental Principles
  of Mathematical Sciences]}.
\newblock Springer-Verlag, Berlin.

\bibitem{Tavaré_BirthProcImm}
{\sc Tavar{\'e}, S.} (1987).
\newblock The birth process with immigration, and the genealogical structure of
  large populations.
\newblock {\em J. Math. Biol.\/} {\bf 25,} 161--168.

\bibitem{Tavaré_2}
{\sc Tavar{\'e}, S.} (1989).
\newblock The genealogy of the birth, death, and immigration process.
\newblock In {\em Mathematical evolutionary theory}.
\newblock Princeton Univ. Press, Princeton, NJ pp.~41--56.

\bibitem{12944964}
{\sc Volkov, I., Banavar, J.~R., Hubbell, S.~P. and Maritan, A.} (2003).
\newblock Neutral theory and relative species abundance in ecology.
\newblock {\em Nature\/} {\bf 424,} 1035--7.

\bibitem{Watterson1974}
{\sc Watterson, G.~A.} (1974).
\newblock Models for the logarithmic species abundance distributions.
\newblock {\em Theoret. Population Biology\/} {\bf 6,} 217--250.

\bibitem{Watterson1974a}
{\sc Watterson, G.~A.} (1974).
\newblock The sampling theory of selectively neutral alleles.
\newblock {\em Advances in Appl. Probability\/} {\bf 6,} 463--488.

\end{thebibliography}
\end{document}